\documentclass[11pt]{amsart}

\usepackage[utf8]{inputenc}
\usepackage[T1]{fontenc}
\usepackage[english]{babel}
\usepackage{fullpage}

\usepackage{amssymb,libertine}

\usepackage{array}
\usepackage{bbm}

\usepackage{mathtools}
\usepackage[svgnames]{xcolor}
\usepackage{hyperref}

\usepackage{wrapfig}
\usepackage{tikz}
\usepackage{color}
\usepackage{graphicx}
\usepackage{amssymb}
\usepackage{epstopdf}
 \usepackage{latexsym}
\usepackage{amsmath}
\usepackage{amsfonts} 
\usepackage{amsthm}
\usepackage[export]{adjustbox}
\def\lap{\mathcal{L}}
   \def\CC{\mathbb{C}}
    
    \def\NN{\mathbb{N}}
    
    \def\RR{\mathbb{R}}
    \def\ZZ{\mathbb{Z}}
    \newtheorem{Proposition}{Proposition}
\newtheorem{Theorem}[Proposition]{Theorem}
\newtheorem{Lemma}[Proposition]{Lemma}
\newtheorem{Definition}[Proposition]{Definition}
\newtheorem{Corollary}[Proposition]{Corollary}

\newtheorem{Note}[Proposition]{Note}
  
\def\be{\begin{equation}}
\def\ee{\end{equation}}
    
\def\ge{\geqslant}

\def\bd{\begin{Definition}}
\def\ed{\end{Definition}}
\def\bt{\begin{Theorem}}
\def\et{\end{Theorem}}

\def\epsilon{\varepsilon}
\def\bel{\begin{equation}\label}
\def\ee{\end{equation}}

\def\Ei\text{Ei}

\def\phi{\varphi}
\title{A new type of factorial expansions}
\author{O. Costin and  R.D. Costin}

\begin{document}
\today 

\maketitle

\begin{abstract}
We construct a new type of convergent asymptotic representations, {\em dyadic factorial expansions}. Their convergence is geometric and the region of convergence can include Stokes rays, and often extends down to $0^+$. For
special functions such as Bessel, Airy, Ei, Erfc, Gamma and others, this region is $\CC$ without an arbitrarily chosen ray effectively providing uniform convergent asymptotic expansions for special functions. 

We prove that relatively  general functions,  \'Ecalle resurgent ones possess convergent  dyadic factorial expansions.  We show that  dyadic expansions are  numerically efficient representations.

The expansions translate into representations of the resolvent of self-adjoint operators in series in  terms of the associated unitary evolution operator evaluated at some prescribed points  (alternatively, in terms of  the generated semigroup for positive operators).

\end{abstract}


\section{Introduction} 
A {\em classical} rising factorial expansion (factorial series) as $x\to\infty$ is a series of the form 
$\displaystyle {  \sum_{k=1}^{\infty}\frac{c_k}{(x)_k}  }$
 where
  \begin{equation}\label{Poch}
 (x)_k:=x(x+1)\cdots (x+k-1)=\frac{\Gamma(x+k)}{\Gamma(x)}
  \end{equation}
 is known as the Pochhammer symbol, or rising factorial. 

Factorial series have a long history { going back to {Stirling}, {Jensen}, {Landau}, {N\"orlund} and {Horn} (see, e.g. \cite{Stirling}, \cite{Jensen}, \cite{Landau}, \cite{N\"orlund}, \cite{Horn})}.  Excellent {introductions} to the classical theory of factorial series and their application to solving ODEs  can be found in the books by N\"orlund \cite{N\"orlund} and Wasow \cite{Wasow}; see also \cite{Paris} Ch.4.

Since  $(x)_{k+1}$ behaves like $k!$ for large $k$, in certain  conditions the factorial expansion of a function converges even when  its {\em asymptotic series} in powers of $1/x$ has empty domain of convergence; we elaborate more on this phenomenon in \S\ref{Classical}.  

Recent use of factorial expansions to tackle  divergent perturbation series in quantum mechanics and quantum field theory (see e.g. \cite{Jen}) triggered considerable renewed interest and substantial  literature. An excellent account of new developments is \cite{Weniger2000}; see  also \cite{Dunster,Adri,We97,JW} and references therein.
\subsection{Drawbacks of classical factorial expansions} Most often, the classical factorial expansions arising in  ODEs and physics have two major limitations: (1) slow convergence, at best power-like; (2) a limited (for the function, unnaturally) domain of convergence: a half plane which cannot be centered on the asymptotically important Stokes ray \footnote{A Stokes ray of a function $f$ is a direction in the Borel $p$ plane along which its Borel  (i.e. formal inverse Laplace) transform $F$ has singularities. If  $\omega$ is a singularity of $F$ then the ray in the $x$ plane $\{x:x\omega\ge 0\}$ is sometimes also called a Stokes ray, and it is the direction where a small exponential is collected in the transseries of $f$. An antistokes ray is a direction where the small exponential becomes classically visible (purely oscillatory).};  see \S \ref{Classical}. As a result they are not suitable for the study of Stokes phenomena (\cite{Jen}, \cite{Borghi}). One aim of the present work is to address and overcome these limitations.

\subsection{Organization of the paper} For clarity of presentation, we start with  examples.  In \S\ref{EiStokes} we first find a geometrically convergent "dyadic" factorial expansion for Ei in  $\CC\setminus i\overline{\RR^-}$,  a region  containing the Stokes ray. In \S\ref{S3} we establish a dyadic decomposition of the Cauchy kernel which we then use in \S\ref{Eiaway} to obtain a somewhat simpler and more efficient expansion of Ei in $\CC\setminus \overline{\RR^+}$. In \S\ref{Bessel} we make  a first step towards generalization and obtain dyadic factorial expansions for Airy and Bessel functions. Further examples and useful identities are given in  \S\ref{Gamma}.

In \S\ref{Resfun}  we develop the general theory of constructing geometrically convergent dyadic expansions  for typical \'Ecalle resurgent functions. Since, by definition,  resurgent divergent series are \'Ecalle-Borel summable (to  resurgent functions, cf. footnote 1),   such  series  are also  resummable in terms of dyadic expansions.

Our theory extends naturally to transseriable functions, but we do not pursue this in the present paper.

In \S\ref{op} we develop dyadic resolvent decompositions for self-adjoint operators in terms of the associated unitary evolution, and, for positive operators, in terms of the evolution semigroup.

In the process, we develop  a general theory of decomposition  of resurgent functions into  simpler resurgent functions, ``resurgent elements''.

\section{Dyadic factorial expansions of Ei  in the Stokes sector}\label{EiStokes}
Let 
\begin{equation}\label{LiEi}
\mathrm{e}^{-x}{\rm Ei}^+(x)=\int_0^{\infty -i0}\frac{\mathrm{e}^{-px}}{1-p}dp
\end{equation}
where $+$ refers to the intended direction of $x$, one in the first quadrant, and by analytic continuation on the Riemann surface of the log. Note that $\RR^+$ is a Stokes ray for $\mathrm e^{-x}$Ei$^+(x)$.

The following identity  holds in $\CC\setminus \{1\}$ (see Corollary\,\ref{dyadic Cauchy} below):
\begin{equation}
  \label{eq:EiS}
 \frac{1}{1-p}= -\frac{\pi  i}{\mathrm{e}^{-i\pi   p}+1}+\pi  i \sum
   _{k=1}^{\infty } \frac1{2^k}\frac{ e_k}{\mathrm{e}^{-r_k
   p}+e_k}\qquad  \text{where } e_k= \mathrm{e}^{-i\pi  2^{-k}},\ r_k=i\pi  2^{-k}
\end{equation} 
Let $x$ be in the first quadrant. We choose the path of integration in \eqref{LiEi} as the vertical segment $[0,-i]$ followed by the horizontal half-line $-i+\RR^+$. Since on this path $|e_k|/|\mathrm{e}^{-r_kp}+e_k|<(1-\mathrm{e}^{-\pi/2})^{-1}$, the functions multiplying $1/{2^k}$ are uniformly bounded and we can Laplace transform the sum term by term. After rescaling $p$ by $2^k$ we get
\begin{equation}\label{intfEi}
 \mathrm{e}^{-x}{\rm Ei}^+(x)=-  i\int_0^{\infty-0i } \frac{ \mathrm{e}^{-p x/\pi}}{\mathrm{e}^{-i  p}+1} \, dp +i\sum_{k=1}^\infty\int_0^{\infty-0i } \frac{e_k\mathrm{e}^{-\frac{2^k p x}{\pi }}}{e_k+\mathrm{e}^{-ip}} \, dp
\end{equation}
Let $x=i\pi y$. After one integration by parts (see also \eqref{eq:Mellin} for changes of variable motivating the way integration by parts is done)  \eqref{intfEi} becomes
$$\mathrm{e}^{-x}{\rm Ei}^+(x)=-\frac{1}{2y}-  \frac iy\,\int_0^{\infty-0i } \frac{ \mathrm{e}^{-ip(y+1)}}{(\mathrm{e}^{-i  p}+1)^2} \, dp +\sum_{k=1}^\infty   \frac{e_k}{2^ky(e_k+1)}+i\sum_{k=1}^\infty \frac{e_k}{2^ky}\, \int_0^{\infty-0i } \frac{e_k\mathrm{e}^{-ip(2^ky+1)}}{(e_k+\mathrm{e}^{-ip})^2} \, dp
$$
and $n-1$ successive integrations by parts yield
\begin{equation}\label{parsum}
\mathrm{e}^{-x}{\rm Ei}^+(x)=-\sum_{m=1}^{n-1}\frac{\Gamma(m)}{2^m(y)_m}\,+R_{n}\,+ \sum_{k=1}^\infty \left( \sum_{m=1}^{n-1}\frac{\Gamma(m)e_k}{(1+e_k)^m(2^ky)_m} \,+R_{nk} \right) 
\end{equation}
where
\begin{equation}\label{remRnRnk}
R_n=-\frac {i\Gamma(n)}{(y)_{n-1}}  \int_{0}^{\infty -i0}\!\frac{{\rm e}^{-ip
 ( y+n-1 ) }{\rm d}p}{ \left(1+ {\rm e}^{-ip} \right) ^{n}},\ \ \text{and }\ R_{nk}=\frac{e_k\Gamma(n)}{(2^ky)_{n-1}} \int_0^{\infty-i0} \frac{\mathrm{e}^{-ip(2^ky+n-1)}}{(e_k+\mathrm{e}^{-ip})^n}  
 \end{equation}
 where the integrals are defined for $y$ in the second quadrant, and the remainders are analytically continued on the Riemann surface of the log.
 
As Proposition \ref{P2} below shows, the remainders go to zero  when $n\to\infty$  and $x \in\CC\setminus -i\overline{\RR^+}$ and we are left with a series which converges geometrically:
\begin{equation}
  \label{eq:Eidsum}
   \mathrm{e}^{-x}{\rm Ei}^+(x)= -\sum_{m=1}^{\infty}\frac{\Gamma(m)}{2^m}\frac{1}{(y)_m}+\sum_{k=1}^\infty\sum_{m=1}^\infty\frac{\Gamma(m) e_k}{(1+e_k)^m}\frac{1}{(2^ky)_m} \ \ \ \ \ \ (y=-i x/\pi)
\end{equation}
 \begin{Proposition}\label{P2} \
 
   (i) For fixed $x\in \CC\setminus -i\overline{\RR^+}$ and large $n$,  $R_n=O(2^{-n}n^{-\Im x/\pi})$. For fixed $n$ and large $x$, $R_n=O(x^{-n})$. 

(ii) For fixed $k$ and $x\in \CC\setminus -i\overline{\RR^+}$, $R_{nk}=O(  |1+e_k|^{-n}n^{-2^k\Im x/\pi})$. For fixed $n$ and large $2^kx$, $R_{nk}=O((2^{k}x)^{-n})$.
\end{Proposition}
\begin{Note}{\rm 
 \textsl{The domain of convergence in Proposition \ref{P2}, a plane with a cut, is clearly larger than the half plane of usual factorial expansions. In fact this domain is maximal for any  convergent meromorphic expansion of Ei$^+$  since, due to the Stokes phenomenon, after a $2\pi$ rotation of $x$ its classical asymptotic behavior changes.   }}
\end{Note}
The numerical efficiency on the Stokes line $\RR^+$, with respect to the number of terms to be kept from each of the infinitely many series in \eqref{eq:Eidsum} can be determined from  Fig. \ref{fig12}. Namely, after choosing a range of $x$ and a target accuracy, one can determine from the graphs the needed order of truncation in each individual series, as well as the number of series as described in Fig. \ref{fig12}.

In Fig. \ref{fig11} we plot the relative error in calculating Ei$^+$ on the Stokes ray.
\begin{figure}
  \centering
\includegraphics[scale=0.4]{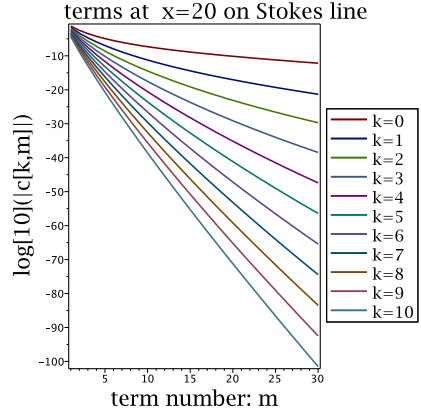}
\vskip -5cm
\caption{Size of terms in the successive series on the Stokes ray $\RR^+$ with the formula {\eqref{eq:Eidsum}}. This plot can be used to determine the number of terms to be kept for a given accuracy. To get $10^{-5}$ accuracy, 10 terms of the first series plus 5 from the second and so on, and all terms from the fifth series on can be discarded.}
\label{fig12}
\end{figure}
\begin{figure}
  \centering
\includegraphics[scale=0.4]{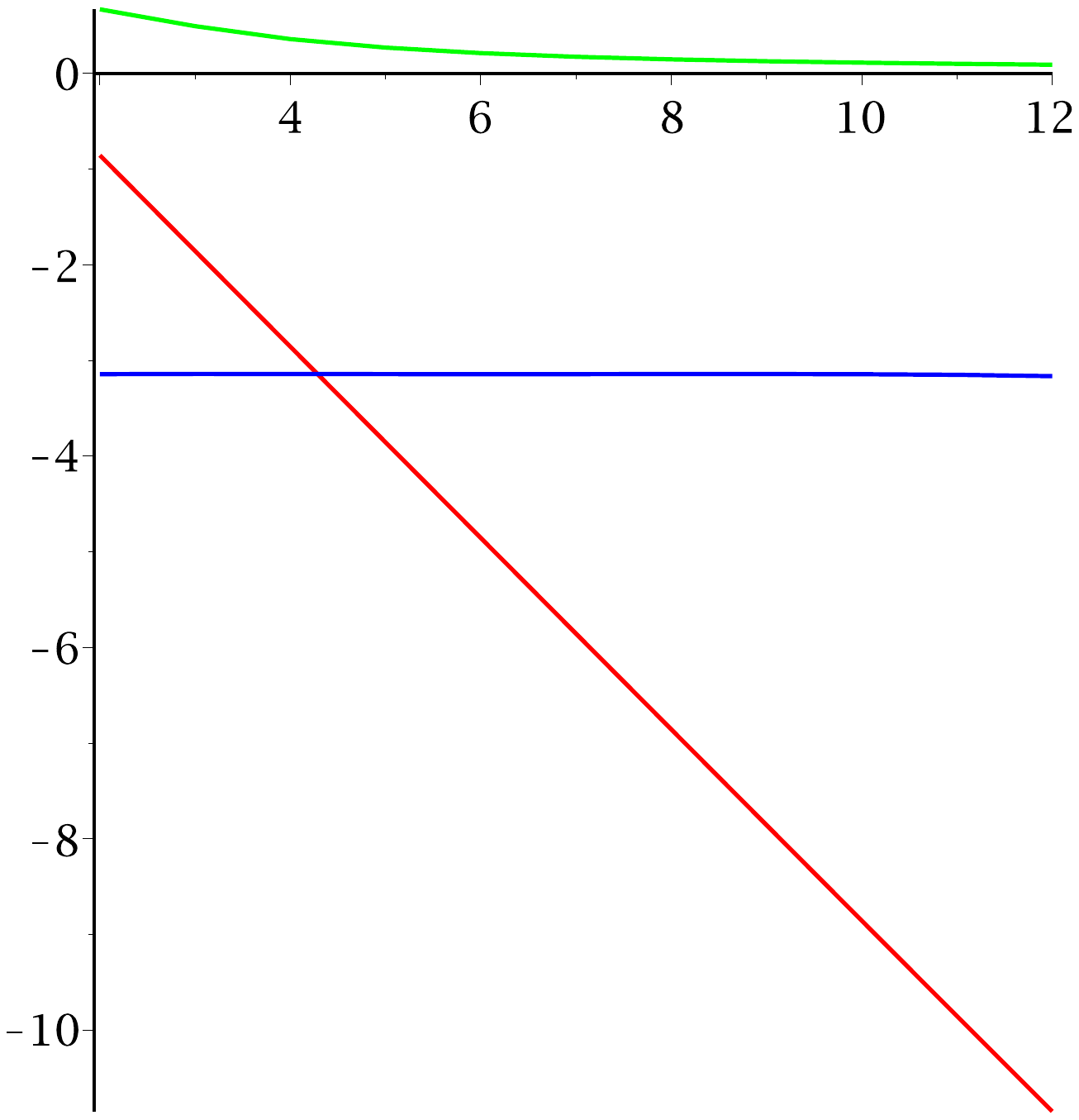}
\vskip -5cm
\caption{$f(x)=\mathrm{e}^{-x}\mathrm{Ei}^+(x)$ on the Stokes line: $\Re f$, (green), $\mathrm{e}^x\Im f$, (blue), $\ln(-\Im f)$, (red), from {\eqref{eq:Eidsum}}. The small exponential is ``born'' on $\RR^+$, with half of the residue, as expected by comparing with $\tfrac12 \mathrm{e}^{-x}\left(\mathrm{Ei}^+(x)+\mathrm{Ei}^-(x)\right)$.  }
\label{fig15}
\end{figure}
\begin{figure}
  \centering
\vskip -1cm 
\includegraphics[scale=0.4]{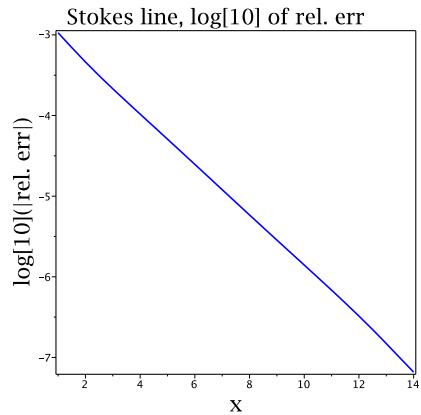}
\vskip -5.2cm
\caption{Numerical errors for $x\in [1,14]$ for $\mathrm{e}^{-x}{\rm Ei}^+(x)$ along the Stokes line with the formula { \eqref{eq:Eidsum}}. }
\label{fig11}
\end{figure}
Figure \ref{st1} below  uses the same expansion \eqref{eq:Eidsum} {for $x$} on the two sides of $-i\RR^+$; in the left picture $\Im \mathrm{e}^{-x}\mathrm{Ei}^+(x)$ is calculated for  $x\in -i\RR-0.3$ and the right one is the graph  of $\Im \mathrm{e}^{-x}\mathrm{Ei}^+(x)$  along $-i\RR+0.3$ (after multiplying by $\mathrm{e}^{0.3}$ to adjust back the size). The oscillatory behavior is due to the exponential (with amplitude $2\pi i$) collected upon  crossing the Stokes ray {$\RR^+$} arg$\,x=-\pi/2$ is an antistokes ray for Ei$^+$).
\begin{Note}
\textsl{There is a dense set of poles in  \eqref{eq:Eidsum}  along $-i\RR^+$  where the expansion breaks down. (This of course does not imply \emph{actual} singularities of {${\rm Ei}^+$}.) Hence, in spite of eventual geometric convergence, near $-i\RR^+$  more and more terms need to be kept  for a given precision.}
\end{Note}
\begin{figure}
  \centering 
\includegraphics[scale=0.3]{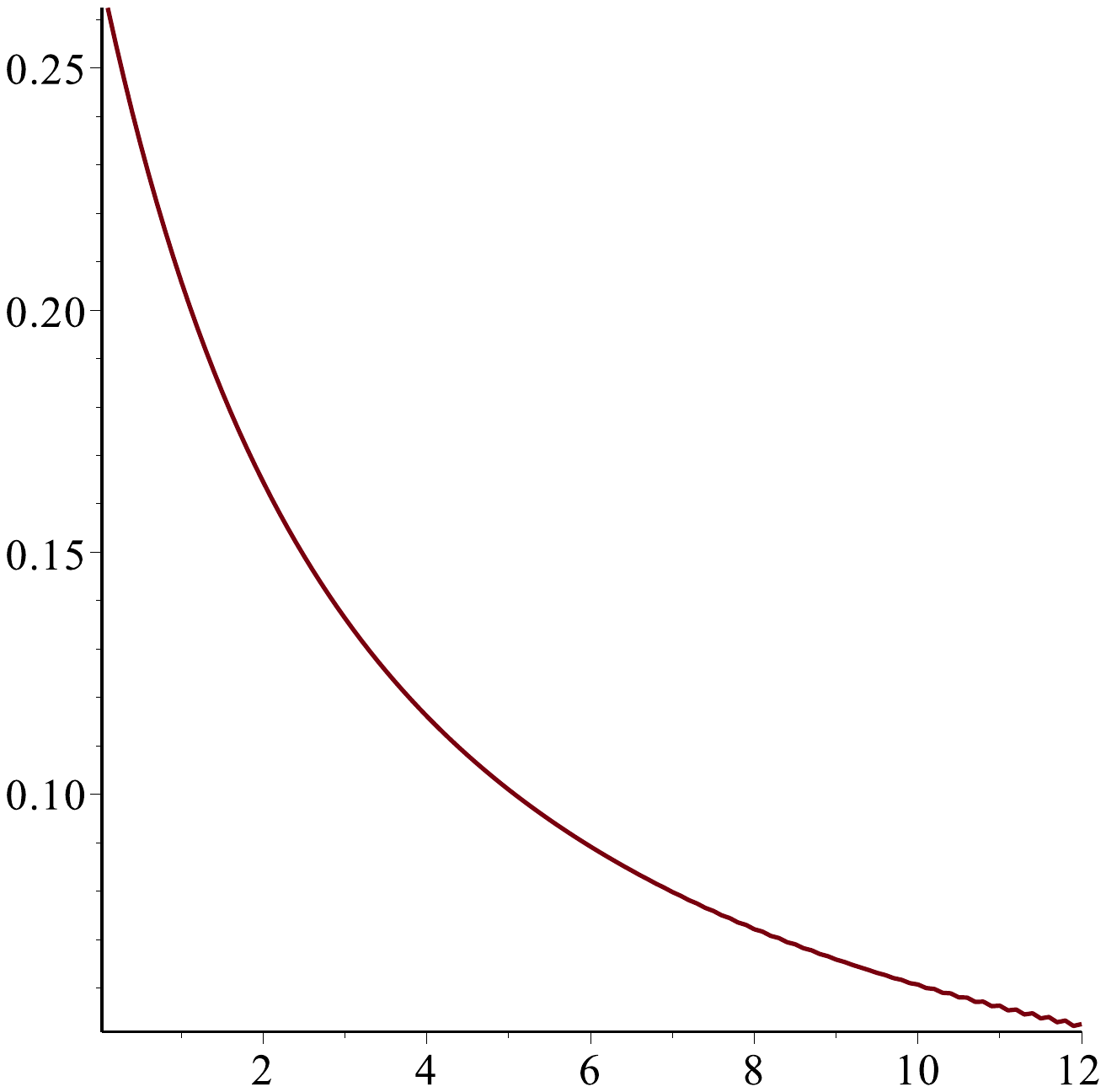} \hskip -2.5cm \includegraphics[scale=0.3]{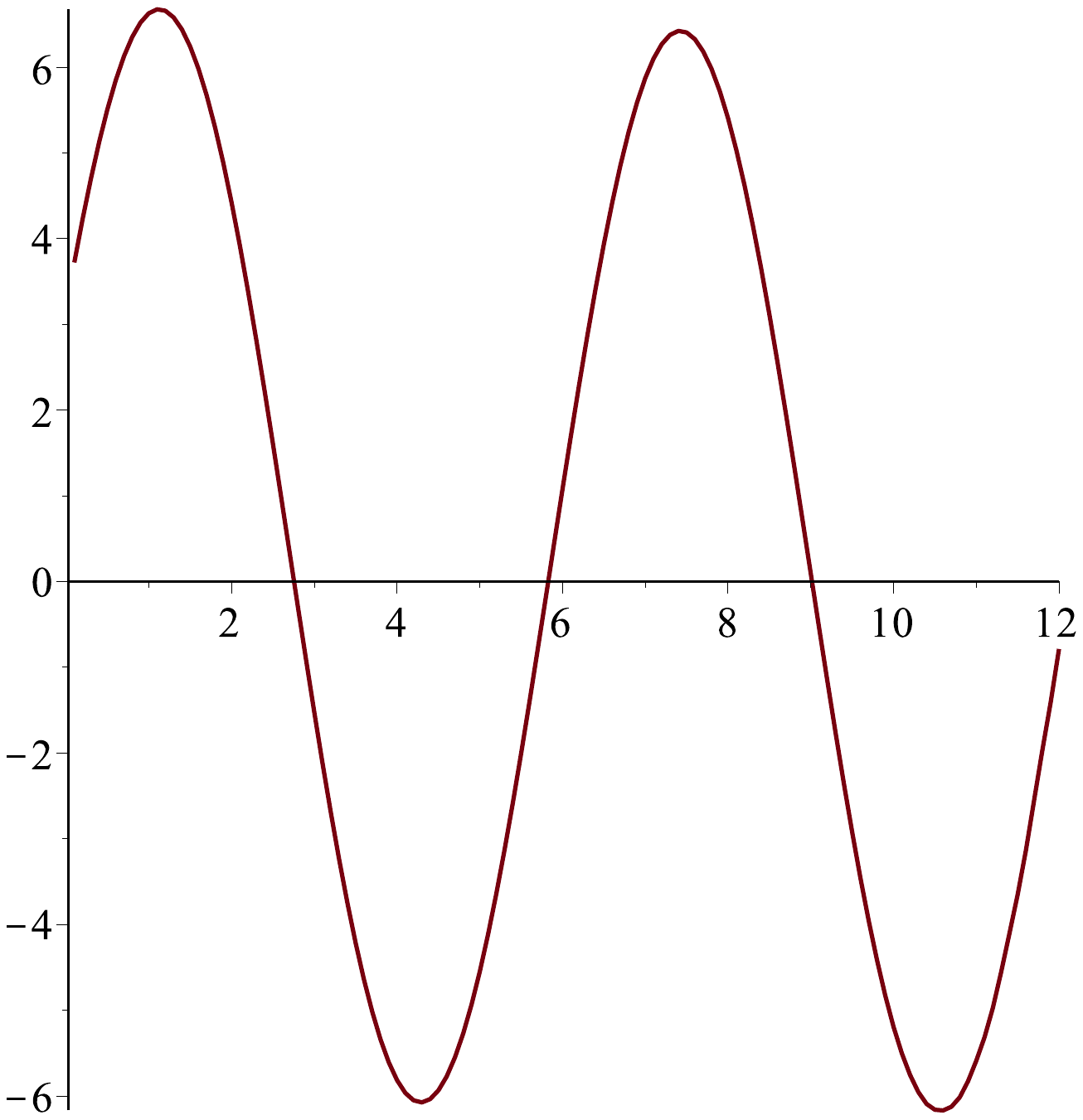}
\vskip -4.0cm
  \caption{The classical Stokes transition of Ei$^+$ from asymptotically decaying to oscillatory.}
\label{st1}
\end{figure}
\begin{proof}[Proof of Proposition\,\ref{P2}]
  The difference between the first integral in \eqref{intfEi} and the first sum in \eqref{eq:Eidsum} truncated to $n-1$ terms is $R_n$  in \eqref{remRnRnk}. 
  
With the notation $\tilde{x}=x/\pi$ (so $y=-i\tilde{x}$) formula  \eqref{remRnRnk} is
\begin{equation}\label{tremRnRnk}
R_n=-\frac {i\Gamma(n)}{(-i\tilde{x})_{n-1}}  \int_{0}^{\infty -i0}\!\frac{{\rm e}^{-p
 ( \tilde{x}+i(n-1) ) }{\rm d}p}{ \left(1+ {\rm e}^{-ip} \right) ^{n}},\ \ \text{and }\ R_{nk}=\frac{e_k\Gamma(n)}{(-i2^k\tilde{x})_{n-1}} \int_0^{\infty-i0} \frac{\mathrm{e}^{-p(2^k\tilde{x}+i(n-1))}}{(e_k+\mathrm{e}^{-ip})^n}  
 \end{equation}
  (i)  For large $n$ we rotate the contour of the integral in  $R_n$   by $-\pi/2$ and change variables to  $q=ip$; the integrand is majorized by $|e^{-qx}| e^{-n[q+\ln(1+e^{-q})]}$. Since $q+\ln(1+e^{-q})$ is increasing,  Laplace's method shows that the integral is $O(n^{-1}2^{-n})$ which combined with Stirling's formula for the prefactor (since $y\not\in\RR^-$) yields the stated estimate. The estimate of $R_{nk}$ is similar.

For fixed $n$ and large $x$ the integral in \eqref{tremRnRnk} is $O(\tilde{x}^{-1})$ by Watson's Lemma, and its prefactor, a multiple of $1/(-i\tilde{x})_{n-1}$, is $O(\tilde{x}^{-n+1})$.

(ii) The proof is similar:  for fixed $k$ and $x$ the integral in \eqref{tremRnRnk}  is  $O(n^{-1}(1+e_k)^{-n})$, while the prefactor is estimated using Stirling's formula. 
 \end{proof}
 \begin{Note}
(i) A variation  {\eqref{eq:deca2} } of \eqref{eq:EiS} allows for optimizing the rate of convergence of Proposition\,\ref{P2}, see Note\,\ref{Note6} part (3).   

(ii) For numerical purposes, for those of the $k$ kept in the calculation which are large enough, the integrals can be evaluated by Berry hyperasymptotics \cite{B1,Bo}.
 \end{Note}
\section{Dyadic decompositions} \label{S3}
\begin{Lemma}[Dyadic decomposition]\label{L1}{\rm 
 \textsl{ The following identity holds in $\CC\setminus \{0\}$:
\begin{equation}
      \label{eq:deca1}
      \frac1p={\frac {1}{1-\mathrm{e}^{-p}}-\sum_{k=1}^\infty{\frac{2^{-k}} {1+\mathrm{e}^{-p/2^k} }}}
    \end{equation}
    (The points $m\pi i$ are removable singularities of the right side: for any $R$, all the finite sums from a certain rank on are analytic in $\{p:|p|\in (0,R)\}$.)\\
\ \ \ The sum converges uniformly on any compact $K\in \CC\setminus \{0\}$.
}}\end{Lemma}
\begin{proof}
 The proof is elementary:
 \begin{equation}
   \label{eq:decx2} \frac{1}{1-x}=\frac{2}{1-x^2}-\frac{1}{x+1}=\frac{4}{1-x^{4}}-\frac{2}{x^2+1}-\frac{1}{x+1}=\ldots=\frac{2^n}{1-x^{2^n}} -\sum_{j=0}^{n-1}\frac{2^j}{1+x^{2^j}} 
 \end{equation}
which implies, with $x=\mathrm{e}^{-p/2^n}$,
\begin{equation}
   \label{eq:decx3} \frac{1}{2^n(1-\mathrm{e}^{-\frac{p}{2^n}})}=\frac{1}{1-\mathrm{e}^{-p}}-\sum_{k=1}^n{\frac{2^{-k}} {\mathrm{e}^{-\frac{p}{2^k}}+1 }}
 \end{equation}
Let now $K\in \CC\setminus \{0\}$ be compact.  In $K$, the left side of \eqref{eq:decx3} converges uniformly  to $p^{-1}$ while the sum on the right side  converges uniformly and absolutely to the series in \eqref{eq:deca1} since for large $k$ (depending on $K$), $|\mathrm{e}^{-\frac{p}{2^k}}+1|^{-1} <2$.
\end{proof}
Let $\beta\ne 0$. The linear affine transformation  $p\to \beta p-\beta s$ gives:
\begin{Corollary}[Dyadic decomposition of the Cauchy kernel]\label{dyadic Cauchy}
\begin{equation}
  \label{eq:deca2}
 \frac{1}{s-p}= {-\frac{\beta  \mathrm{e}^{-\beta s}}{\mathrm{e}^{- \beta s }-\mathrm{e}^{-\beta p}   }  }+ {\sum _{k=1}^{\infty } \frac{\beta 2^{-k} \mathrm{e}^{-2^{-k}  \beta s}}{\mathrm{e}^{-2^{-k}
   \beta s}+\mathrm{e}^{-2^{-k} \beta  p}}}
\end{equation}
The series converges uniformly for $p$ in compact sets in $\CC\setminus\{s\}$.
\end{Corollary}
See Notes\, \ref{Note6} and \ref{nn15} for the importance of the choice of the parameters $s$ and $\beta$.
\begin{Note}\label{Note6}
 {\rm  
    \begin{enumerate}
    \item  The dyadic factorial expansion \eqref{eq:EiS} of Ei is in fact obtained using the general formula \eqref{eq:deca2} with $\beta=1/2$, $s=2\pi i$ and $2\pi ip$ instead of $p$. 
    \item The choice of direction $s\in i\RR^+$   is crucial for the result in Proposition\,\ref{P2}.

    \item As mentioned, there is some arbitrariness in the choice of $\beta$ and $ s$ in \eqref{eq:deca2}. In particular the choice above leads to diminishing the effective variable from $x$ to $y=-ix/\pi$, see \eqref{eq:Eidsum}. Depending on the range of $x$, other choices of $\beta$ would lead to better convergence rates, cf. Proposition \ref{PP6w}.

 \end{enumerate}
}\end{Note}
\section{ Ei away from the Stokes ray, in $\CC\setminus \RR^+$}\label{Eiaway}
In \S\ref{EiStokes} we used dyadic expansions to obtain geometrically convergent expansions for Ei in $\CC\setminus -i\RR^+$, a region containing the Stokes ray $\RR_+$. In this section we revisit the problem of obtaining somewhat simpler and more efficient expansions (faster than $2^{-m}$) away from the Stokes ray.

There is substantial literature on classical factorial series representations of the exponential integral in the left-half plane, which is {\em the sector opposite} to the Stokes ray. For an excellent account of the literature see the recent paper \cite{Weniger2000}. See also  \cite{nist}(6.10) for extensive references.

Rotating the line of integration in \eqref{LiEi} clockwise by an angle $\pi^-$  while rotating $x$ clockwise, we see that the study of Ei in $\CC\setminus -i\RR^+$ is equivalent to the study for  $0\ne x\in\CC, \arg\,x\ne\pi$ of the function
$$ {\rm e}^x{\rm Ei}(-x)=\int_0^{\infty}\frac{\mathrm{e}^{-xp}}{p+1}dp $$
\begin{Proposition}\label{PP6w}{\rm 
  The following identity holds for all $0\ne x\in\CC, \arg\,x\ne\pi$:
  \begin{equation}
    \label{eq:iden2}
{\rm e}^x{\rm Ei}(-x)= \int_0^\infty \frac{{\rm e}\,\,\mathrm{e}^{-xp}}{\mathrm{e}-\mathrm{e}^{-p}}dp -\sum_{k=1}^{\infty}\int _0^\infty\frac{\mathrm{e}^{2^{-k}}\mathrm{e}^{-2^k xq}}{\mathrm{e}^{2^{-k}}+\mathrm{e}^{-q}}\,{dq }  \\
=\Phi(\mathrm{e}^{-1},1,x-1)-\sum_{k=1}^{\infty}\Phi(-\mathrm{e}^{-2^{-k}},1,2^kx)
  \end{equation}
where $\Phi$ is the Lerch Phi transcendent. Relation \eqref{eq:iden2} implies 
\begin{equation}\label{eq:iden22}
{ \mathrm{e}}^x{\rm Ei}(-x)=\sum _{m=1}^{\infty }  \frac { (-1)^{m+1}{\mathrm{e}}
 \Gamma(m)}{ (\mathrm{e}-1)^m(x)_m} -\sum _{k=1}^{\infty }\sum _{m=1}^{\infty } \left( {\frac {\Gamma(m)\,{\mathrm{e}^{2^{-k}}}}{ \left( {
\mathrm{e}^{2^{-k}}}+1 \right) ^{m}{({2}^{k}x)_m}}}\right)
\end{equation}
The remainders, defined as in \eqref{parsum}, satisfy: For fixed $y\in \CC\setminus \overline{\RR^-}$, $R_n=O((\mathrm{e}-1)^{-n}n^{-\Re x})$. For fixed $n$ and large $y$,  $R_n=O(x^{-n})$. For fixed $x\in \CC\setminus \overline{\RR^-}$,  $R_{kn}=O( 2^{-k}  (1+\mathrm{e}^{2^{-k}})^{-n}n^{-2^k\Re y})$. For fixed $k,n$ and large $x$, $R_{kn}=O(2^{-k}x^{-n})$.}
\end{Proposition}
\begin{Note}{\rm 
 The effective variable, $2^k x$, gets rapidly large for large $k$ and not many  terms of the double sum are needed in practice. Even for $x=0.1$ the first sum above requires 20 terms to give $10^{-5}$ relative errors. 
    } 
\end{Note}
\begin{proof}[Proof of Proposition \ref{PP6w}] 
 Taking $\beta=1$ and $s=-1$ in \eqref{eq:deca2}, applying the Laplace transform, and then using dominated convergence we obtain \eqref{eq:iden2}. See \cite{nist}(25.14.5) for the integral representation of the Lerch function.

Repeated integration by parts as in \S\ref{EiStokes} yields the expansion \eqref{eq:iden22}; now, the remainders are:
\begin{equation}
  \label{eq:eq112}
 R_n= \frac{(-1)^{n}\mathrm{e}\Gamma(n)}{(x)_{n-1}}\int_0^{\infty}\frac{\mathrm{e}^{-p(x+n-1)}}{(\mathrm{e}-\mathrm{e}^{-p})^{n}}dp=\frac{(-1)^{n}\mathrm{e}\Gamma(n)\Gamma(x)}{\Gamma(x+n-1)}\int_0^{\infty}\frac{\mathrm{e}^{-p(x+n-1)}}{(\mathrm{e}-\mathrm{e}^{-p})^{n}}dp
\end{equation}
and 
\begin{equation}
  \label{eq:eq113n}
R_{nk}=-e^{2^{-k}}  \frac{\Gamma(n)}{(2^kx)_{n-1}}\int_0^{\infty}\frac{\mathrm{e}^{-q(2^kx+n-1)}}{(\mathrm{e}^{2^{-k}}+\mathrm{e}^{-q})^{n}}dq=-e^{2^{-k}}  \frac{\Gamma(n)\Gamma(2^kx)}{\Gamma(2^kx+{n-1})}\int_0^{\infty}\frac{\mathrm{e}^{-q(2^kx+n-1)}}{(\mathrm{e}^{2^{-k}}+\mathrm{e}^{-q})^{n}}dq
\end{equation}
The remainders are estimated as in the proof of Proposition\,\ref{P2}. 
  \end{proof}
\section{Dyadic expansions for Airy and Bessel functions}\label{Bessel}
To our knowledge, the first systematic study of classical factorial series for Bessel function is \cite{Dunster}; see \cite{Dunster2} for subsequent developments.
\subsection{Dyadic expansions for the Airy function Ai}
Again, to keep the logic simple, we analyze in some detail the Airy function Ai, as the general Bessel functions are dealt with similarly, as explained in \S\ref{genBessel}. 

After normalization, described in \S\ref{sec5}, the asymptotic series of the Airy function is Borel summable:
\begin{equation}
  \label{eq:eqh1}
  h(x)=\int_0^{\infty}\mathrm{e}^{-px}F(p)dp
\end{equation}
where $F(p)={_2F_1}(1/6,5/6;1,-p)=P_{-1/6}(1+2p)$ is analytic except for a  logarithmic singularity at $-1$, see \eqref{eq:solna} and \eqref{eq:eqlog} below. The decay of $F$ 
for large $p$ is relatively slow, $O(p^{-1/6})$, and we integrate once by parts to improve it:
\begin{equation}
  \label{eq:intpts}
  h(x)=\frac{F(0)}{x}+\frac1x\int_0^{\infty}\mathrm{e}^{-px}F'(p)dp
\end{equation}
 and use Cauchy's formula (needed for applying the derivative of \eqref{eq:deca2})
\begin{equation}
  F'(p)=\frac{1}{2\pi i}\oint_{|p-s|<r}\frac{F(s)}{(s-p)^2}ds=\frac{1}{2\pi i}\int_{-\infty}^{-1}\frac{\Delta F(s)}{(p-s)^2}ds
\end{equation}
where we pushed the contour to  infinity so that a subsequent Laplace contour, $\RR^+$ does not intersect the $s$ integral, see Note \ref{nn15} below. We are left with a Hankel contour around $-1$; $\Delta F$ is the jump of $F$  across the cut $(-\infty,-1)$.
\begin{Note}\label{nn15}\textsl{
When using \eqref{eq:deca1}, to be able to interchange summation and integration in a contour integral, we need of course to ensure that each term and not merely the sum in \eqref{eq:deca1} is analytic on the contour. If $s$ parametrizes the curve, then in particular the curve must avoid the half-lines $s=p+k\pi i/\beta$. If $F$ has only one singularity, then the  Cauchy formula contour can be deformed into a H\"ankel-like curve towards $-\infty$ for a suitable $\beta$.}
\end{Note}
After the change of variables $s=-1-t$ we get
\begin{equation}
  \label{eq:Airy3}
 h(x)=  \frac{F(0)}{x}-\frac{1}{2\pi x} \int_0^{\infty}\mathrm{e}^{-xp}\int_{0}^{\infty}\frac{F(t)}{(1+p+t)^2}dt
\end{equation}
since $\Delta F(-1-t)=-iF(t)$, see \eqref{w+-w-} below.

To use the expansion \eqref{eq:deca2} in \eqref{eq:Airy3} we first differentiate \eqref{eq:deca2} in $p$ and take $\beta=1$ obtaining
$$\frac 1{(s-p)^2}=\frac{\mathrm{e}^{-s-p}}{(\mathrm{e}^{-s}-\mathrm{e}^{-p})^2} +\sum_{k=1}^\infty 4^{-k} \frac{ \mathrm{e}^{2^{-k}(-p-s)}}{(\mathrm{e}^{-2^{-k}s}+\mathrm{e}^{-2^{-k}p})^2}$$
which for $s=-1-t$ yields
\begin{equation}
  \label{eq:int2}
\int_{0}^{\infty}\frac{F(t)dt}{(1+p+t)^2}=\int_{0}^{\infty}\frac{\mathrm{e}^{1-p+t}F(t)dt}{(\mathrm{e}^{1+t}-\mathrm{e}^{-p})^2}+\sum_{k=1}^{\infty}\int_{0}^{\infty}\frac{4^{-k}\mathrm{e}^{2^{-k}(1-p+t)}F(t)dt}{(\mathrm{e}^{2^{-k}(1+t)}+\mathrm{e}^{-2^{-k}p})^2}
\end{equation}
yielding
\begin{multline}
  h(x)=  \frac{F(0)}{x}-\frac{1}{2\pi x} \int_0^{\infty} \mathrm{e}^{-p(x+1)} dp\int_{0}^{\infty}\frac{\mathrm{e}^{1+t}F(t)dt}{(\mathrm{e}^{1+t}-\mathrm{e}^{-p})^2}   \\
 - \frac{2^{-k}\mathrm{e}^{2^{-k}}}{2\pi x} \int_0^{\infty} \mathrm{e}^{-q(2^kx+1)}dq \sum_{k=1}^{\infty}\int_{0}^{\infty}\frac{\mathrm{e}^{2^{-k}t}F(t)dt}{(\mathrm{e}^{2^{-k}(1+t)}+\mathrm{e}^{-q})^2}
\end{multline}
The dyadic factorial series is obtained, as before for Ei,  by repeated integration by parts, integrating the exponentials. This yields the dyadic factorial expansion 
\begin{equation}
  \label{eq:32}
 h(x)=\frac{F(0)}{x}-\sum_{m=2}^{\infty}\frac{(-1)^{m}\Gamma(m)}{2\pi  (x)_{m}}d_m-\sum_{k=1}^{\infty}2^{-k}{\rm e}^{2^{-k}}\sum_{m=2}^{\infty}\frac{\Gamma(m)}{2\pi  (2^kx)_{m}}d_{km}
\end{equation}
where 
\begin{equation}
  \label{dmdkm}
d_m:=\int_0^{\infty}\frac{F(t)\mathrm{e}^{t+1}dt}{(\mathrm{e}^{t+1}-1)^{m}};\ \ \ \ \  d_{km}:=\int_0^{\infty}\frac{{\rm e}^{2^{-k}t}F(t)dt}{(\mathrm{e}^{2^{-k}(1+t)}+1)^{m}}
\end{equation}
Unlike in the case of Ei however, the coefficients $d_m$ do not have a simple closed form expression, nor of course can this be expected in general. The integrals can be evaluated numerically, or by power series. Alternatively, they can be calculated in the $x$ domain.  Indeed, with $\varphi=\mathcal{L}F$, 
\begin{equation}
  \label{eq:sum3}
\int_0^{\infty}\frac{F(t)\mathrm{e}^{t+1}dt}{(\mathrm{e}^{t+1}-1)^{m+2}}=\sum_{j=0}^{\infty}\mathrm{e}^{-m-j-2} \int_0^{\infty}\mathrm{e}^{-(m+j+1)t}F(t)dt\\=\mathrm{e}^{-m-1}\sum_{j=1}^{\infty}\mathrm{e}^{-j}
\binom{m+j}{j}\frac{\varphi(m+j)}{m+j}
\end{equation}
and for Airy, $\varphi=h$.

Fig.\,\ref{Fig5}  shows the numerical results from \eqref{eq:32} using Mathematica in machine precision to evaluate the integrals in the $d_m,d_{km}$.
\subsection{General Bessel functions} \label{genBessel}

There are  few and relatively minor adaptations needed to deal with $K_\nu$ for more general $\nu$. After normalization, explained in \S\ref{sec5}, $F(p)$ is now the Legendre function $P_{\nu-1/2}(1+2p)$ for which the branch  jump at $-1$ is   $\Delta F(-1-p)=-2i\cos (\pi  \nu)F(p)$ (see \eqref{w+-w-}) and the leading behavior at infinity is $O(p^{|\Re \nu|-1/2})$. The steps followed in the Airy case apply after integrating by parts  $k$ times until $|\Re \nu|-1/2-k<-1$. Alternatively, one can apply the general transformation  in \S\ref{TT1} that ensures exponential decay. For $J_{\nu},Y_{\nu}$ the procedure is the same, except that the singularity is now on the imaginary line. For $J_\nu$ the singularity is on $\RR^+$ and a choice of $\beta$ as for Ei$^+$ needs to be made.
\begin{figure}
  \centering 
\includegraphics[scale=0.43]{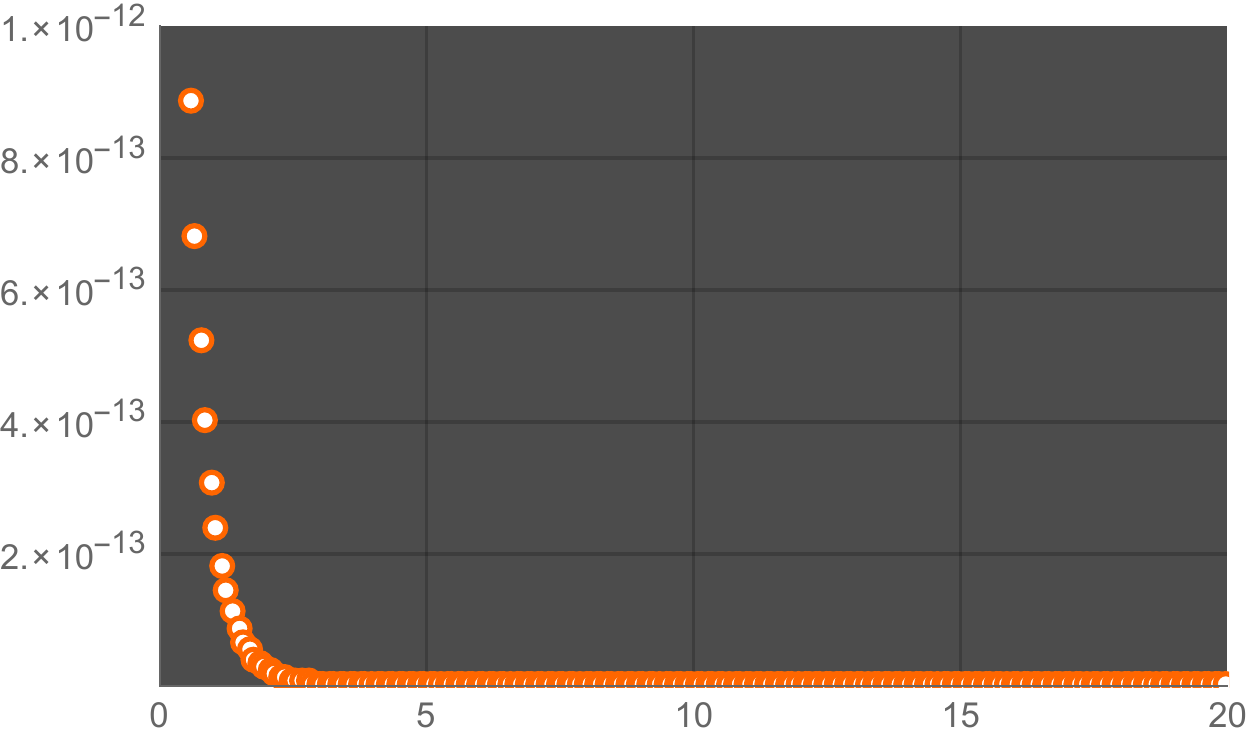}  \includegraphics[scale=0.43]{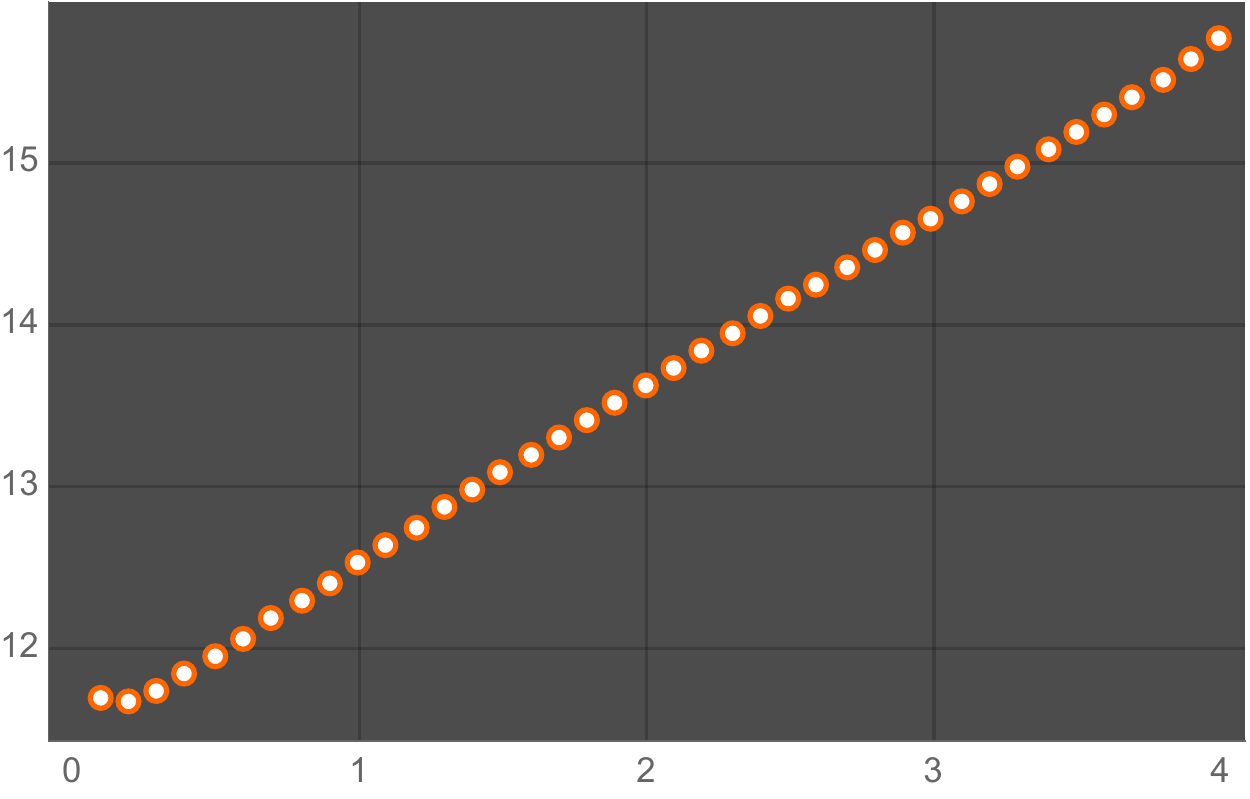} 
\caption{Relative accuracy for Ai (left), and number of exact digits (right) as functions of $x$. The total number of terms used in this calculation ranges from about 150 for small $x$  to 30 terms at $x=20$, found as explained in Fig. \ref{fig12}. The right graph plateaus at 16 digits for all  $x\ge 4$,  an artefact due to calculations being made in  Mathematica's  machine precision; thus the right graph was stopped at $x=4$. }
\label{Fig5}
\end{figure}
\section{Dyadic resolvent identities}\label{op}
Dyadic decompositions translate into representations of the resolvent of a self-adjoint operator in a series involving the unitary evolution operator at specific discrete times:
\begin{Proposition}\label{C8}\textsl{ 
(i) Let $\mathcal{H}$ be a Hilbert space,  and $A$ a bounded or unbounded self-adjoint operator. Let $U$ be the unitary evolution operator generated by $A$, $U_t=\mathrm{e}^{-itA}$. If $\lambda\in\RR^+$, then}
\begin{multline}
  \label{eq:resol1}
(A-i\lambda)^{-1}  =i(1-\mathrm{e}^{-\lambda}U_1)^{-1}-i\sum_{k=1}^{\infty}\frac1{2^k}(1+e^{-{\lambda}/2^{k}}U_{2^{-k}})^{-1} \\=i\sum_{j=0}^{\infty}\mathrm{e}^{-j\lambda}U_{j} -i\lim_{\ell\to\infty}\sum_{k=1}^{\ell}\sum_{j=0}^{\infty}(-1)^j\mathrm{e}^{-{j\lambda}/{2^k}}U_{j2^{-k}}
\end{multline}
\textsl{Convergence holds in the strong operator topology. For $\lambda<0$ one simply complex conjugates \eqref{eq:resol1}.  (The limits cannot, generally, be interchanged.)}

 (ii) Assume $A$ is a  positive operator (thus self-adjoint) and  $0\notin\sigma(A)$. Let $T_t$ be the semigroup generated by $A$, $T_t=\mathrm{e}^{-tA}$. Then 
 \begin{equation}
   \label{eq:Semigr}
A^{-1}=(1-T_1)^{-1}-\sum_{k=1}^{\infty}2^{-k}(1+T_{1/2^k})^{-1}=\sum_{j=1}^{\infty}T_{j}-\lim_{\ell\to\infty}\sum_{k=1}^{\ell}\sum_{j=1}^{\infty}2^{-k}(-1)^jT_{j/2^k}
 \end{equation}
 \textsl{where now convergence is in  operator norm. More generally, for $ s<1$, $s\notin\ZZ$,}
\begin{equation}
  \label{eq:li21}
 \pi A^{s-1}=\Gamma(s)\sin(\pi s)\left[ \mathrm{Li}_s\left(T_1\right) - \sum _{k=1}^{\infty} 2^{-k (1-s)} \mathrm{Li}_s\left(-T_{1/2^k}\right)\right]
\end{equation}
\textsl{in operator norm. Here, for $|z|<1$, the polylog is defined by}
\begin{equation}
  \label{eq:defLi}
  \mathrm{Li}_s(z)=\sum_{k=1}^{\infty}k^{-s}z^k
\end{equation}
\end{Proposition}
\begin{proof}
(i) We recall the projector-valued measure  spectral theorem for  self-adjoint operators. If $\mathcal{H}$ and $A$ are as above  and  $g:\RR\to\RR$ is a  Borel function (or  a complex one, by writing $g=g_1+ig_2$), then $g(A)=\int_{-\infty}^{\infty}g(q) dP_{q}$ where $\{P_{\Omega}\}$ are the projector-valued measures induced by $A$ on $\mathcal{H}$
 (see \cite{ReedAndSimon} Theorem VIII.6 p. 263). The spectral theorem together with  \eqref{eq:decx3}   for  $p=\lambda+iq$ give 
 \begin{equation}
   \label{eq:decop}
   (1-\mathrm{e}^{-\lambda}U_1)^{-1}-\sum_{k=1}^n2^{-k}(1+ \mathrm{e}^{-2^{-k}\lambda}U_{2^{-k}})^{-1}=\epsilon(1-\mathrm{e}^{-\lambda\epsilon}\mathrm{e}^{-i\epsilon A})^{-1}=\int_{\RR}\frac{\epsilon dP_q}{1-\mathrm{e}^{-\epsilon(\lambda+iq)}}
 \end{equation}
where  $\epsilon_n:=\epsilon=2^{-n}$. An elementary calculation shows that the modulus of the integrand is uniformly bounded by $\lambda^{-1}$. Since the integrand converges pointwise to $(\lambda+iq)^{-1}$ as $\epsilon\to 0$, dominated convergence shows that the integral converges to $(\lambda+iA)^{-1}$. Dominated convergence  also shows that the integrand, seen as a multiplication  operator, converges in the strong operator topology, implying the result.

(ii) The proof,  based on the same argument as in (i), is simpler  and we omit it. For \eqref{eq:li21} we combine this argument with  Lemma \ref{L1L} below. The sums in \eqref{eq:Semigr} are manifestly convergent in the operator norm since  $\|T_t\|<1$ and $T_t>0$.
  \end{proof}
\section{When do classical factorial series converge geometrically?}\label{Classical}

Here we motivate the treatment of general resurgent functions in \S\ref{Resfun}  and explain why expansions of the form \eqref{eq:deca2} yield to geometrically convergent factorial expansions. The conclusions are summarized in 
Note\,\ref{note4}.

The connection of Horn factorial expansions to Borel summation is made already in \cite{N\"orlund}. Assume $f$ is the Borel sum of a series, that is 
\begin{equation}\label{fisLF}
f(x)=\int_0^\infty F(p)\mathrm{e}^{-px}\, dp
\end{equation}
 where $F$ is analytic in an open sector containing $\RR^+$ and  exponentially bounded at infinity. The asymptotic series for large $x$ follows from Watson's lemma \cite{Wasow} or, in this case, simply by integration by parts: for $x$ large enough we have
\begin{equation}
  \label{serfF}
  f(x)=x^{-1}F(0)+x^{-2}F'(0)+\cdots+x^{-n}F^{(n-1)}(0)+x^{-n}\int_0^\infty F^{(n)}(p)\mathrm{e}^{-px}\, dp
\end{equation} 
Integration by parts results in a growing power of $\frac{d}{dp}$ and thus, by Cauchy's theorem leads to factorial divergence of the asymptotic series, unless $F$ is entire (rarely the case in applications). N\"orlund notices however that the simple change of variables $\phi(s)=F(-\ln s)$ brings the representation of $f$  to the form 
\begin{equation}
  \label{eq:Mellin}
  f(x)=\int_0^1s^{x-1}\phi(s)ds
\end{equation}
Now integration by parts gives the {\em factorial} expansion
\begin{equation}
  \label{eq:fct2}
 f(x)= \phi(1)\, \frac{1}x-\phi'(1)\, \frac{1}{(x)_2}+\cdots +\frac{(-1)^{n-1}}{(x)_n}\,\phi^{(n-1)}(1) +\frac{(-1)^n}{(x)_n}\int_0^1 s^{x+n-1}\phi^{(n)}(s)ds
\end{equation}
or, without remainder, we have the \emph{factorial series}  (Horn expansion)
\begin{equation}
  \label{eq:fct2s}
\tilde{f}(x)= {\sum_{k=0}^{\infty}(-1)^k\frac{\phi^{(k)}(1)}{(x)_{k+1}}}
\end{equation}
\begin{Note}{\rm 
Since $F$ is analytic at zero, $\phi$ is analytic at one. Using Stirling's formula in \eqref{Poch}, we see that, for large $k$,  the $k+1$th term of the expansion \eqref{eq:fct2} behaves like 
\begin{equation}\label{ktermphi}
  (-1)^k\, {\Gamma(x)} \,\frac{ \phi^{(k)}(1)\ }{k!}\, k^{-x}  
\end{equation}
Due to the $1/k!$ factor in \eqref{ktermphi} the series $\tilde{f}(x)$ can converge even if \eqref{serfF} is factorially divergent.

}\end{Note}
\begin{Note}{\rm 
For {$\tilde{f}$}  to converge, {\eqref{ktermphi} shows that} $\phi$ needs to be analytic in a disk of radius one centered at {$s=1$}, which translates in analyticity of $F$ in the region ${\mathcal{S}=}\big\{p=p_1+ip_2|\, |p_2|<\arccos (\frac12 \mathrm{e}^{-p_1})\big\}$; $\mathcal{S}$ is a strip-like region of {width} $\pi$ centered on $\RR_+$ (see \cite{Wasow} p. 328). 

 If $F$ is analytic in a strip $\omega \mathcal{S}$ (for some $\omega\ne 0$) and $F$ has at most exponential growth, then replacing $p$ by $\omega p$ and $x$ by $x/\omega$ then, {in the new variables,}  the conditions mentioned before are {\textsl{necessary and sufficient}} for convergence of \eqref{eq:fct2s}. 

}\end{Note}
\begin{Note}(i) {\rm{We also note that if $F$ is exponentially bounded in a strip $\mathcal{S}$ then {$\tilde{f}(x)$} converges in a half-$x$-plane, and no more. A rigorous proof based on Hadamard's theory of order on the circle of convergence is given in \cite{N\"orlund} pp.  45-59. Heuristically, if $|F(p)|\lesssim \mathrm{e}^{\beta|p|}$ in $\mathcal{S}$ then $|\phi(s)|\lesssim |s^{-\beta}|$ near $s=0$ and then, by Cauchy's formula, 
$|\phi^{(k)}(1)| / {k!}\lesssim k^{\beta-1}$ hence convergence requires $\Re x>\beta$. }

{\em (ii)} Conversely, if a factorial series ${f}(x)=\sum \frac{c_k}{(x)_k}$ converges in a half-plane, since
\begin{equation}
  \label{eq:xtop}
  \left[\mathcal{L}^{-1}\frac{1}{(x)_k}\right](p)=\frac{(1-\mathrm{e}^{-p})^k}{k!}
\end{equation}
(see \S\ref{RfL}) we have
$$[\mathcal{L}^{-1}f](p)= \sum_{k=0}^{\infty} \, c_k\,\frac{(1-\mathrm{e}^{-p})^k}{k!}:=F(p)$$
with $F$ as in (i).  (See also \cite{N\"orlund} p.188.)                            
}\end{Note}
\begin{Note}{\rm {As mentioned, $F$ is required to be analytic on $\RR^+$; hence $\RR^+$ cannot be a Stokes line (a line containing Borel plane singularities). Because of this, {\em classical factorial series \eqref{eq:fct2} are not suitable for the study of Stokes phenomena} (\cite{Jen}, \cite{Borghi})}. 
\begin{Note}\label{note4} {\rm Convergence of{ $\tilde{f}(x)$ }is typically slow, generally at most power-like. The theorems in \cite{N\"orlund} and \cite{Wasow} are too general to allow for {more precise} estimates of the rate. In the specific case of Bessel functions of order $\nu$, Lutz and Dunster \cite{Dunster} showed that the $m$th term in the series is ${m!}/{(z)_{m+1}}\, O(m^{-1}\ln^{\nu-5/2}m)$ for $\Re z>\epsilon>0$, implying that the rate of convergence is  $O(m^{-\Re z}\ln m^{\nu-5/2})$. In general, by \eqref{ktermphi}, we see that convergence is geometric only if $\phi$ is analytic in a disk of radius $>1$, in particular at $s=0$, or analytic in $s^\beta$ by changing $x$ to $2+(x-2)/\beta$, or, more generally, a sum of such functions for various $\beta$'s. Therefore convergence cannot be geometric unless $F$ is analytic in $\mathrm{e}^{-\beta p}$ at infinity (for some $\beta$), or a sum of such functions.
}\end{Note}
}\end{Note}
\section{Dyadic series of general resurgent functions}\label{Resfun}  
In a nutshell, a {\em resurgent function} in the sense of \'Ecalle is  a function which is endlessly continuable and has suitable exponential bounds at infinity \cite{Ecalle}. The singularities are typically assumed to be regular, in the sense of having convergent local Puiseux series possibly mixed with logs. 

In this paper we restrict to functions which appear in generic meromorphic ODEs and difference equations. More precisely, {\em a resurgent function is a function $F$ (in Borel plane)  which: has a finite number of arrays of  singularities;  in each  array the singularities are regular and equally spaced;  and $F$ is exponentially bounded at infinity (away from the singular arrays)}. For details see  \cite{Duke} and \cite{Braaksma}. By abuse of language, the Laplace transform of a resurgent function is often also called ``resurgent''.

Using Lemma \ref{L1}, we show that modulo simple, algorithmic transformations, resurgent functions can be written  in the form  $\sum_j F_j(\mathrm{e}^{-\beta_j p})$, where the sum converges geometrically,   $F_j(\mathrm{e}^{-\beta_j p})$ are also resurgent, and  $F_j(z)$ are analytic at zero. Thus the factorial series of $\lap F_j(\mathrm{e}^{-\beta_j p})$ are geometrically convergent (see Note\,\ref{note4}) in a cut plane, thus allowing for the study of Stokes phenomena. Due to rapid convergence, their associated factorial series are also suitable for precise and efficient numerical calculations.
\subsection{Elementary resurgent functions}  We define {\em resurgent ``elements'' to be
resurgent functions with only one regular singularity, and with algebraic decay at infinity}.  
There are two main properties of resurgent elements which do not  hold for general resurgent functions: decay at infinity in $p$ and the property of having only one singularity. However, the following decomposition holds:
\begin{Theorem}\label{TT1}
  The Laplace transform of a resurgent function as described at the beginning of \S\ref{Resfun} can be written, modulo a convergent series at infinity and translations of the variable, as a sum of Laplace transforms of resurgent elements. 
\end{Theorem}
\begin{proof}
  The proof is given in \S\ref{PT}.
\end{proof}
\begin{Note}  The exponential integral and the $\Psi$ function treated in \S\ref{Psi} are examples of elements with nonramified singularities. Airy and Bessel functions treated in \S\ref{Bessel} are examples of elements with ramified singularities, treated via the Cauchy kernel decomposition. The incomplete gamma function and the error function  treated in \S\ref{erfc} have power-ramified singularities for which a polylog dyadic expansion (Lemma\,\ref{L1L}) gives more explicit decompositions. Theorem\,\ref{TT1} extends these techniques to general resurgent functions.
\end{Note}
\subsection{Proof of Theorem \ref{TT1}} \label{PT}
In this section we describe how a general resurgent function can be decomposed into resurgent elements. To avoid cumbersome details and keep the presentation clear, we present the essential steps in the case where the resurgent function is the Laplace transform of a solution of a generic meromorphic ODE.

Denoting the singularities of the resurgent function $F$ by $\omega_i$, we thus assume:
\begin{enumerate}\renewcommand{\theenumi}{\alph{enumi}}
\item Each $\omega_i$ is of the form $j\lambda_k$, with $j\in\ZZ^+$ and $\lambda_k\in\{\lambda_1,\ldots,\lambda_n\} $ (the  eigenvalues of the linearization at $\infty$ of the ODE assumed to be linearly independent over $\ZZ$ and of different complex arguments);
\item there is a $\nu$ such that
$$\|F\|_\nu : =\sup_{p\in \mathcal{A}}|F(p)\mathrm{e}^{-\nu |p|}|<\infty$$
 where $\mathcal{A}$ is  the complement of the union of thin half-strips $S_i$ containing exactly one singularity $\omega_i$. We let $\mathcal{C}_i=\partial S_i$,  (see Fig. \ref{fig16}).  $\mathcal{C}_i$ are  are non-intersecting H\"ankel contours around the $\omega_i$, going vertically if $\omega_i$ belongs to a singularity ray in the open right half plane and towards $\infty$ in the left half plane otherwise; $\mathcal{C}_i$ are traversed anticlockwise. 
\end{enumerate}
\begin{figure}
  \centering
\hspace{2cm}\includegraphics[scale=0.7]{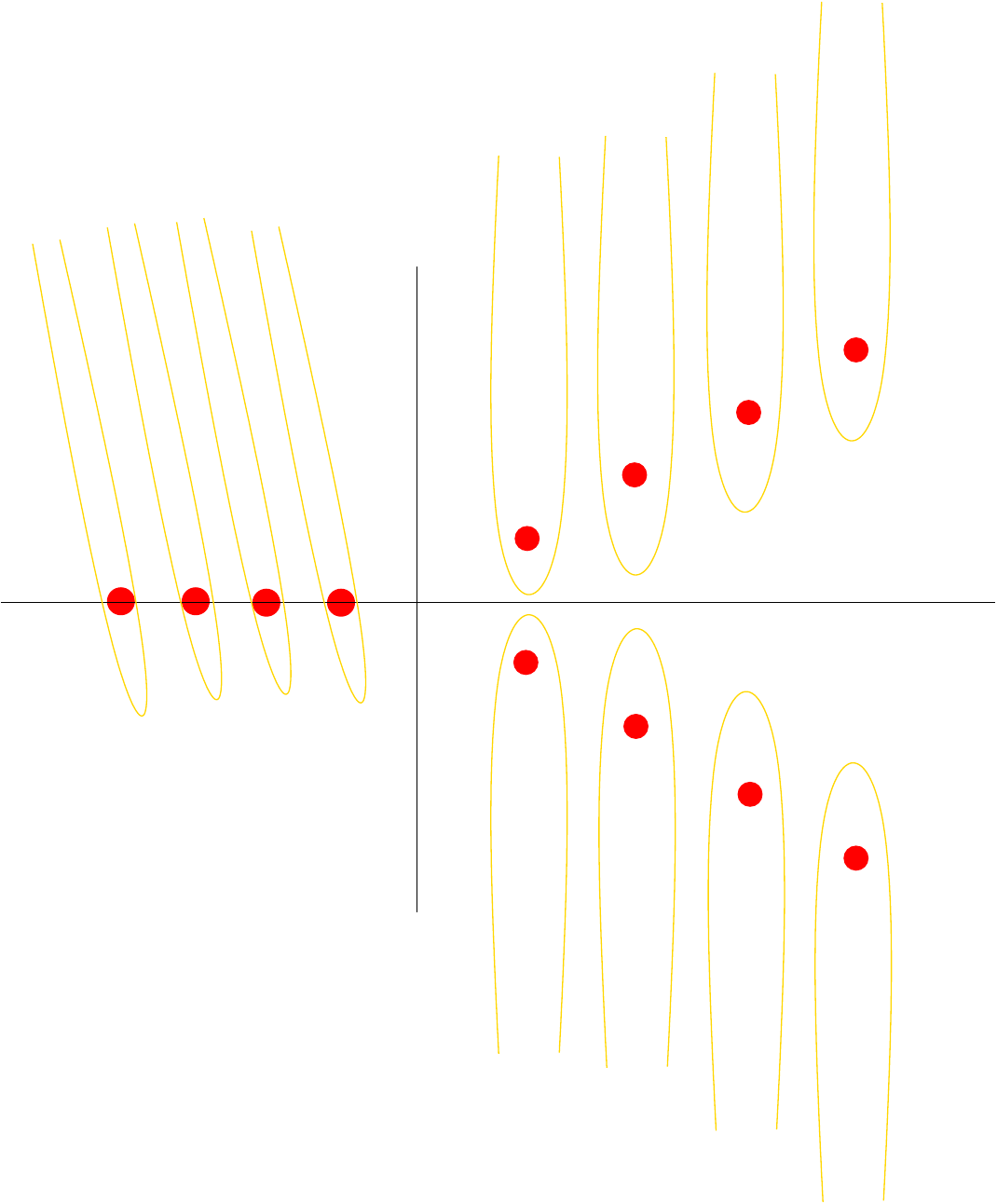}
\vskip -0.5cm
\caption{The H\"ankel contours $\mathcal{C}_i$.}
\label{fig16}
\end{figure}
 Let
  \begin{equation}
    \label{eq:decp}
    G(p)=F(p)-\sum_{\omega_i} \frac{\exp(\mu_i p)}{2\pi i}\int\limits_{\mathcal{C}_i} \frac{F(s)\exp(-\mu_i s)}{s-p}ds
  \end{equation}
where:
\begin{enumerate}
\renewcommand{\theenumi}{\alph{enumi}}
\item  $|\mu_i|=\mu>\nu$, 
\item  $\arg(\mu_i)$ is the negative of the angle of the contour $\mathcal{C}_i$, i.e., $\mu_i s\in\RR^+$ for large $s$. Note that  the set $\{\arg\mu_i\}$ is finite, since there are only finitely many rays with singularities.
\end{enumerate}
\begin{Lemma}
In any compact set in $\mathcal{A}$, the sum in \eqref{eq:decp} converges at least as fast as $\sum_{j\in\ZZ^+,k=1,...,n}\mathrm{e}^{-j|\lambda_k|(\mu-\nu)}$. 
\end{Lemma}
\begin{proof}
Let  $\omega_i=j\lambda_k$ and
\begin{equation}
  \label{eq:F_i}
  F_i(p)=\frac{\exp(\mu_i p)}{2\pi i}\int\limits_{\mathcal{C}_i} \frac{F(s)\exp(-\mu_i s)}{s-p}ds
\end{equation}
For some $a$ depending on the width of the contour, the distance to the singularity (these two parameters can be chosen to be the same for all contours), on the position and diameter of the compact set (also the same for all $i$),  we have 
$$ \left|F_i\right|\leqslant a\, \|F\|_{\nu}\,\mathrm{e}^{-j|\lambda_k|(\mu-\nu)}$$
\end{proof}
\begin{Lemma}
  On the first Riemann sheet, each $F_i$ in \eqref{eq:F_i} has precisely one singularity, namely at $\omega_i$. Furthermore $F-F_i$ is analytic at $\omega_i$.
\end{Lemma}
\begin{proof}
  Let  $p\ne \omega_i$.  If $p$ is outside $\mathcal{C}_i$ then  function $F_i$ is manifestly analytic at $p$. To analytically continue in $p$ to the interior of $\mathcal{C}_i$  it is convenient to first deform $\mathcal{C}_i$ past $p$, collecting the residue. We get
  $$ F_i(p)=\frac{\exp(\mu_i p)}{2\pi i}\left[ \int\limits_{\tilde{\mathcal{C}}_i} \frac{F(s)\exp(-\mu_i s)}{s-p}ds+2\pi i F(p) \exp(-\mu_i p)\right]=F(p)+\frac{\exp(\mu_i p)}{2\pi i}\int\limits_{\tilde{\mathcal{C}}_i} \frac{F(s)\exp(-\mu_i s)}{s-p}ds$$
where now $p$ sits inside $\tilde{\mathcal{C}}_i$, and the new integral is again manifestly analytic. 

Thus $F_i$ is singular only at $p=\omega_i$, and $F-F_i$ is analytic at $\omega_i$.
\end{proof}
\begin{Lemma}
  The function
  \begin{equation}
    \label{eq:eqG}
    G(p)=F(p)-\sum_{i}F_i
  \end{equation}
is entire and $\|G\|_{\mu'}<\infty$ for any $\mu'>\mu$.
\end{Lemma}
\begin{proof}
  Analyticity follows from the monodromy theorem, since $G$ has analytic continuation along any ray in $\CC$. The bound follows easily from the previous lemmas.
\end{proof}
\begin{Lemma}
  $g=\mathcal{L}G$ has a convergent asymptotic series at infinity, and is equal to the sum of the series. 
\end{Lemma}
\begin{proof}
  Cauchy estimates show in a straightforward way that $G^{(n)}(0)\lesssim \mu'\,^n/n!$.  Watson's lemma shows convergence of the series. The function $h(z)=g(1/z)$ is bounded at zero and single-valued, as is seen by deformation of contour (since $G$ is exponentially bounded and entire). Thus $h$ is analytic at zero, and therefore the sum of its asymptotic (=Taylor) series at zero.
\end{proof}
\begin{Lemma}
  Each function $\mathrm{e}^{-\mu_i p} F_i$ decays like $1/p$ as $p\to\infty$. 
\end{Lemma}
\begin{proof}
  The function $pF_i$ is manifestly bounded.
\end{proof}
\begin{Lemma}
  The change of variable $\tilde{x}=x-\mu_i$ leads to
$\lap[F_i](x)=\lap[\tilde{F}_i](\tilde{x})$ where $\tilde{F}_i$ decays like $1/p$ as $p\to\infty$. 
\end{Lemma}
Combining these lemmas, Theorem\,\ref{TT1} follows. 
  \section{The $\Psi$ function}\label{Gamma}
   \subsection{Dyadic factorial expansion for the $\Psi$ function}\label{Psi}
Replacing $p$ by $-p$ in \eqref{eq:deca1} we get
\begin{equation}
  \frac1p-\frac{1}{\mathrm{e}^p-1}=\sum_{k=1}^\infty\frac {{\rm e}^{-\frac {p}{{2}^{k}}}}{2^k\left( {{\rm e}^{
-{\frac {p}{{2}^{k}}}}}+1 \right)} 
\end{equation}
On the other hand  we have, see \cite{Book} eq. (4.61) p. 99,
\begin{equation}
  \label{eq:lng}
  \frac{\Gamma'(x+1)}{\Gamma(x+1)}-\ln x=\int_0^\infty\left(\frac1p-\frac{1}{\mathrm{e}^p-1}\right)\mathrm{e}^{-xp}dp
\end{equation}
Thus, changing the variable of integration to $q=p/2^k$ we get
\begin{equation}
  \label{eq:gamma2}
 \Psi(x+1)= \frac{\Gamma'(x+1)}{\Gamma(x+1)}=\ln x+\sum_{k=1}^{\infty}\int_0^\infty\frac{\mathrm{e}^{-q(1+2^k x)}}{1+\mathrm{e}^{-q}}\, dq
\end{equation}
and integrating by parts we obtain the dyadic factorial expansion
\begin{equation}
  \label{eq:doublexp}
  \Psi(x+1)=\ln x-\sum_{k=1}^{\infty}\Phi(-1,1,2^kx+1)=\ln x+\sum_{k=1}^{\infty}\sum_{j=1}^{\infty}{\frac{(j-1)!}{2^j(2^k x+1)_j}}
\end{equation}
\subsection{Factorial expansion for differences of the $\Psi$ function and a strange identity}\label{DPsi}
\begin{Proposition}\label{curious}
We have
\begin{equation}
  \label{eq:modif}
\frac{1}{2} \Psi \left(\frac{x}{2}+\frac{1}{2}\right)-\frac{1}{2} \Psi
   \left(\frac{x}{2}\right)= \int_0^1\frac{t^{x-1}}{t+1}dt = \frac{1}{2x} \ - \ \frac{1}{2^2(x)_2 }+\cdots+\frac{(-1)^{n-1}\Gamma(n)}{2^n(x)_n}+\cdots
\end{equation}
Combining with \eqref{eq:doublexp} we get
\begin{equation}
  \label{eq:ide}
  \Psi \left( x+1 \right) =\ln  x  -\frac12 \sum _{k=0}^{
\infty }\left[\Psi \left( {2}^{k}x+1 \right) -\Psi \left( {2}^{k}x+\tfrac12\right)\right]
\end{equation}
\end {Proposition}
\begin{proof}
Consider the functional equation
\begin{equation}
  \label{eq:difdif1}
  f(x+1)+f(x)=\frac 1x  
\end{equation}
After Borel transform (i.e. substitution of \eqref{fisLF}) it becomes $(\mathrm{e}^{-p}+1)F(p)=1$ yielding
\begin{equation}
  \label{eq:id4}
  f(x)=\int_0^\infty\,\frac{\mathrm{e}^{-px}}{\mathrm{e}^{-p}+1}\, dp= \int_0^\infty\,\sum_{n=0}^\infty\,(-1)^n\,\mathrm{e}^{-p(x+n)}\, dp = \sum_{n=0}^\infty \frac{(-1)^n}{x+n}
\end{equation}
where the interchange of summation and integration is justified, say, by the monotone convergence theorem applied to $\sum_{n=0}^{2N}(-1)^n\mathrm{e}^{-p(x+n)}$. Of course, the integral converges only for $\Re x>0$, but the series converges for all $x\notin\{0,-1,-2,...\}$. Therefore $f(x)$ is meromorphic, having simple poles at $x=-n,\ n\in\NN$.

 On the other hand $f(x)=\tfrac12\psi(\tfrac x2+\tfrac12)-\tfrac12\psi(\tfrac x2)$ which follows from integrating the  identity
$$\psi'(z)= \sum_{n=0}^{\infty}(z+n)^{-2}$$
 (see \cite{Ahlfors},  (31) p. 200) between $z=\frac x2$ and $z=\frac{x+1}2$. 
 
 The integral representation in \eqref{eq:modif} then follows by substituting $\mathrm{e}^{-p}=t$ in \eqref{eq:id4} and the factorial expansion in \eqref{eq:modif} is then obtained as usual, by integration by parts.
 
 \end{proof}
 \subsection{Duplication formulas and incomplete Gamma functions}\label{dupl}
The polylog $\displaystyle{  \mathrm{Li}_{s}\left(z\right)=\sum_{k=1}^\infty \frac{z^k}{k^s}}$ has the integral representation	
\begin{equation}
  \label{eq:intrep}
  \mathrm{Li}_{s}\left(z\right)=\frac{z}{\Gamma
\left(s\right)}\int_{0}^{\infty}\frac{x^{s-1}}{\mathrm{e}^{x}-z}dx
\end{equation}
and satisfies the general duplication formula
\begin{equation}
  \label{eq:dupl}
  f(z)+f(-z)=2^{1-s}f(z^2)
\end{equation}
(see \cite{Lewin}; also, \eqref{eq:intrep}, \eqref{eq:dupl} are easily checked directly).
\begin{Lemma}[A ramified generalization of \eqref{eq:deca1}]\label{L1L}{\rm 
 \textsl{ The following identity holds in $\CC\setminus \{0\}$ if $s<1$:
\begin{equation}
  \label{eq:li2}
 \pi p^{s-1}=\Gamma(s)\sin(\pi s)\left[ \mathrm{Li}_s\left(\mathrm{e}^{-p}\right) - \sum _{k=1}^{\infty} 2^{-k (1-s)} \mathrm{Li}_s\left(-\mathrm{e}^{-2^{-k} p}\right)\right]
\end{equation}
which reduces to \eqref{eq:deca1} if $s=0$.
}}\end{Lemma}
\begin{proof}
 Let $s<1$. As in the proof of Lemma\,\ref{L1} we iterate \eqref{eq:dupl} $n$ times, then replace $z$ by $e^{2^{-n}z}$ to obtain
\begin{equation}
  \label{eq:Li}
  \frac{\mathrm{Li}_{s}\left(\mathrm{e}^{z/2^n}\right)}{2^{n(1-s)}}=\mathrm{Li}_{s}\left(\mathrm{e}^{z}\right)-\sum_{k=1}^{n}2^{-k(1-s)}\mathrm{Li}_{s}\left(-\mathrm{e}^{z/2^k}\right)
\end{equation}
\begin{equation}
  \label{eq:li222}
 2^{-n
   (1-s)}\text{Li}_s\left(\mathrm{e}^{-2^{-n} z}\right)= \text{Li}_s\left(\mathrm{e}^{-z}\right) - \sum _{k=1}^n 2^{-k (1-s)} \text{Li}_s\left(-\mathrm{e}^{-2^{-k} z}\right)
\end{equation}
 Using the identity \cite{nist}(25.12.12)
$$\mathop{\mathrm{Li}_{s}\/}\nolimits\!\left(z\right)=\mathop{\Gamma\/}\nolimits%
\!\left(1-s\right)\left(\mathop{\ln\/}\nolimits\frac{1}{z}\right)^{s-1}+\sum_{%
n=0}^{\infty}\mathop{\zeta\/}\nolimits\!\left(s-n\right)\frac{(\mathop{\ln\/}%
\nolimits z)^{n}}{n!},\ \ \ \ \ \ \ \ \ \ \ \ s\ne 1,2,\ldots,\ |\ln z|<2\pi$$
in \eqref{eq:li222} we get, in  the limit $n\to\infty$,
$$ z^{s-1}\Gamma(1-s)= \text{Li}_s\left(\mathrm{e}^{-z}\right) - \sum _{k=1}^{\infty} 2^{-k (1-s)} \text{Li}_s\left(-\mathrm{e}^{-2^{-k} z}\right)  $$
from which \eqref{eq:li2} follows by using the reflection formula $\Gamma(s)\Gamma(1-s)=\pi/\sin(\pi s)$. 
\end{proof}
\subsection{Dyadic factorial series for incomplete gamma functions and erfc}\label{erfc} The incomplete gamma function is defined by
$$\Gamma(s,x)=\int_x^\infty \, t^{s-1}\, \mathrm{e}^{-t}\, dt$$
and has as a special case the error function, 
$$\text{erfc}(x)=\frac{2}{\sqrt{\pi}}\int_x^\infty\, \mathrm{e}^{-t^2}\, dt=\frac{1}{\sqrt{\pi}}\, \Gamma\left(\frac 12,x^2\right)$$ 
Noting that
$$\int_0^{\infty}(1+p)^{s-1}\mathrm{e}^{-xp}dp=\mathrm{e}^x x^{-s}\Gamma(s,x)$$
we see that $\mathrm{e}^x x^{-s}\Gamma(s,x)$ is the Laplace transform of a function which has a ramified singularity if $s\not\in\ZZ$. In this case we apply Lemma\,\ref{L1L} and obtain the expansion, for $s<1$
\begin{equation}
  \label{eq:gammat}
  \Gamma(1-s)\mathrm{e}^xx^{-s}\Gamma(s,x)=\mathcal{L} \,\text{Li}_s\left(\mathrm{e}^{-p-1}\right)-\sum _{k=1}^{\infty} 2^{-k (1-s)} \mathcal{L}\,\text{Li}_s\left(-\mathrm{e}^{-2^{-k} (p+1)}\right)
\end{equation}
and in particular
\begin{equation}
  \label{eq:erfc}
   \pi e^x x^{-1/2}\text{erfc}\left(\sqrt{x}\right)=\mathcal{L}\,\text{Li}_{\frac{1}{2}}\left(\mathrm{e}^{-p-1}\right)-\sum
   _{k=1}^{\infty } 2^{-k/2} \mathcal{L}\,\text{Li}_{\frac{1}{2}}\left(-\mathrm{e}^{-2^{-k}
   (p+1)}   \right)
\end{equation}
From this point on, the dyadic expansions are obtained as in the previous examples. For example, the first Laplace transform in \eqref{eq:erfc} has the factorial series
$$ \mathcal{L}\,\text{Li}_{\frac{1}{2}}\left(\mathrm{e}^{-p-1}\right)=\int_0^1t^{x-1}\text{Li}_{\frac{1}{2}}\left(\frac t{\mathrm{e}}\right)=\sum_{k=0}^\infty\frac{(-1)^k}{\mathrm{e}^k(x)_{k+1}}\text{Li}_{\frac{1}{2}}^{(k)}\left(\mathrm{e}^{-1}\right):=\sum_{k=0}^\infty\frac{c_k}{(x)_{k+1}} $$
with
$$ c_k =(-1)^k\sum_{j=0}^k s(k,j)\text{Li}_{\frac{1}{2}-j}\left(\mathrm{e}^{-1}\right) $$
where $s(k,j)$ are the Stirling numbers of the first kind, where we used the formula (see \S\ref{form} for details)
\begin{equation}\label{derLi}
\frac{d^k}{dz^k}\text{Li}_{\nu}\left(z \right)=z^{-k}\sum_{j=0}^k s(k,j)\text{Li}_{\nu-j}(z)
\end{equation}
\section{Appendix}
\subsection{The rising factorial and the Laplace transform}\label{RfL}
Let $\delta f (x)=f(x)-f(x+1)$. We first note that 
\begin{equation}
  \label{eq:invld}
 \left( \mathcal{L}^{-1}\delta f\right)(p) =(1-\mathrm{e}^{-p})  \left(\mathcal{L}^{-1} f\right)(p) 
\end{equation}
and that $\displaystyle  \delta ({1}/{(x)_n})= n/{(x)_{n+1}} $, hence
\begin{equation}
  \label{eq:eqid4}
 \delta^n \frac{1}{x}= \frac{n!}{(x)_n}
\end{equation}
Eq. \eqref{eq:eqid4} and \eqref{eq:invld} imply \eqref{eq:xtop}.
\subsection{Normalized Airy and Bessel functions}\label{sec5}
The modified Bessel equation is 
\begin{equation}
  \label{eq:eqb}
  x^2y'' +xy'-(\nu^2+x^2)y=0
\end{equation}
The transformation 
$y=\mathrm{e}^{-x}x^{1/2}h(2x)$, $u=2x$ brings \eqref{eq:eqb} to the normalized form  
\begin{equation}
  \label{eq:eqnormB}
  h''-\left(1-\frac2u\right)h'-\left(\frac1u-\frac1{4u^2}+\frac{\nu^2}{u^2}\right) h=0
\end{equation}
This normalized form is suitable for Borel summation since it admits a formal power series solution in powers of $u^{-1}$ starting with $u^{-1}$; it is further normalized to ensure that the Borel plane singularity is placed at $p=-1$. One way to obtain the transformation is to rely on the classical asymptotic behavior of Bessel functions and seek a transformation that formally leads to a solution as above. 

The Airy equation 
\begin{equation}
  \label{eq:Ai1}
 f''-xf=0
\end{equation}
can be brought to the Bessel equation with $\nu=1/3$, as is well known. The normalizing transformation can be obtained directly by the recipe above, based on the asymptotic behavior at $\infty$.  With the change of variable
$$f(x)=x^{5/4}\mathrm{e}^{-\frac23 x^{3/2}}h(x);\ x=(3u/4)^{2/3}$$
the equation becomes
\begin{equation}
  \label{eq:eqa1}
  h''-\left(1-\frac2{u}\right)h'-\left(\frac1{u}-\frac{5}{36u^2}\right)h=0
\end{equation}
which is indeed \eqref{eq:eqnormB} for $\nu=1/3$. From this point, without notable algebraic complications we analyze \eqref{eq:eqnormB}.

The inverse Laplace transform of \eqref{eq:eqnormB} is
\begin{equation}
  \label{eq:borel}
  p (p+1) H''(p)+(2 p+1) H'(p)+\left(\frac{1}{4}-\nu ^2\right) H(p) =0
\end{equation}
whose solution which is {\em analytic} at zero is (a constant multiple of)
\begin{equation}
  \label{eq:solna}
 \ _2F_1\left(\tfrac12+\nu,\tfrac12-\nu;1;-p\right)=P_{\nu-\frac12}(1+2p)
\end{equation}
where $\,_2F_1$ is the usual hypergeometric function and $P$ is the Legendre $P$ function \cite{nist}(14.3.1). On the first Riemann sheet, the solution has two regular singularities, $p=-1$ and $p=\infty$. The behavior at zero is \cite{nist}(15.2.1)
$$ P_{\nu-\frac12}(1+2p)=1+\left(\nu ^2-\frac{1}{4}\right) p+\frac{1}{64} \left(16 \nu ^4-40 \nu
   ^2+9\right) p^2+\cdots$$
At $\infty$, the convergent series of the solution is \cite{nist}(15.12.1(i))
$$ P_{\nu-\frac12}(1+2p)=\Gamma(-2\nu)\Gamma\left(\tfrac{1}{2}-\nu\right)^{-2}p^{-\nu-1/2}[1+o(1/p)]
+\Gamma(2\nu)\Gamma\left(\nu+\tfrac{1}{2}\right)^{-2}p^{\nu-1/2}[1+o(1/p)]$$
At the singularity $p=-1$ we have (see \cite{nist}(14.8.2) and (14.6.1))
\begin{equation}
  \label{eq:eqlog}
 P_{\nu-\frac12}(1+2p)=  -\pi^{-1}\cos (\pi  \nu ) \log (p+1)A_1(1+p)+A_2(1+p)
\end{equation}
where $A_1,A_2$ are analytic and $A_1(-1)=1$.

The difference between the analytic continuation of $P_{\nu-\frac12}(1+2p)$ below and above $(-\infty,-1)$ is (see  \cite{nist}(15.10), or directly, from \eqref{eq:eqlog} and the fact that the difference of two solutions of \eqref{eq:borel} is again a solution) 
\begin{equation}\label{w+-w-}
P_{\nu-\frac12}(1+2p)^--P_{\nu-\frac12}(1+2p)^+=-2i\cos(\pi\nu)P_{\nu-\frac12}(-1-2p)
\end{equation}
  \subsection{The derivatives of the polylogarithm}\label{form}
  For $k=0$ we have $s(0,0)=1$. It is easy to check that Li$_s'(z)=z^{-1}$Li$_{s-1}(z)$ confirming that $s(1,0)=0$ and $s(1,1)=1$. For higher $k$ formula \eqref{derLi} is then checked by a simple induction, which leads to the recurrence relations
 $$s(k+1,0)=-ks(k,0),\ \ s(k+1,k+1)=s(k,k),\ \ s(k+1,j)=-ks(k,j)+s(k,j-1)$$
 which are the recurrence relations satisfied by the Stirling numbers of the first kind, see \cite{nist} Sec.26.8.
 \section{Acknowledgments}
The first author was partially supported by the NSF grant DMS - 1515755.

\end{document}